\documentclass{amsart}
\usepackage{amsmath,amssymb,amsfonts,bbold}
\usepackage[mathscr]{eucal}

\usepackage{enumerate}

\theoremstyle{plain}
\newtheorem{theorem}{Theorem}[section]

\newtheorem{prop}[theorem]{Proposition}

\theoremstyle{definition}
\newtheorem{definition}[theorem]{Definition}
\newtheorem{example}[theorem]{Example}
\newtheorem{remark}[theorem]{Remark}
\newtheorem*{thank}{Acknowledgments}

\numberwithin{equation}{section}

\input xy
\xyoption{all}

\newcommand{\ad}{\operatorname{ad}}

\newcommand{\colim}{\operatorname{colim}}

\newcommand{\Deltaop}{\Delta^{\op}}

\newcommand{\hocolim}{\operatorname{hocolim}}

\newcommand{\holim}{\operatorname{holim}}

\newcommand{\Hom}{\operatorname{Hom}}

\newcommand{\Map}{\operatorname{Map}}

\newcommand{\modcat}{\operatorname{mod}}

\newcommand{\op}{\operatorname{op}}
\newcommand{\Sets}{\mathcal Sets}
\newcommand{\SSets}{\mathcal{SS}ets}

\newcommand{\Sp}{\mathcal Sp}
\newcommand{\Spc}{\operatorname{Spc}}

\begin{document}

\title[Homotopy limits]{Homotopy limits of model categories, revisited}

\author[J.E. Bergner]{Julia E. Bergner}

\address{Department of Mathematics, University of Virginia, Charlottesville, VA 22904}

\email{jeb2md@virginia.edu}

\date{\today}



\thanks{The author was partially supported by NSF grant DMS-1906281.}

\begin{abstract}
The definition of the homotopy limit of a diagram of left Quillen functors of model categories has been useful in a number of applications.  In this paper we review its definition and summarize some of these applications.  We conclude with a discussion of why we could work with right Quillen functors instead, but cannot work with a combination of the two.
\end{abstract}

\dedicatory{To John Greenlees on the occasion of his 60th birthday}

\maketitle

\section{Introduction}

In the papers \cite{fiberprod} and \cite{holim}, we develop constructions for the homotopy pullback, and then more general homotopy limit, of a suitable diagram of left Quillen functors between model categories.  In this paper, we seek to revisit some of the ideas of those papers, with three main goals.  First, we would like to try to make some of the exposition of those papers more accessible to the reader.  Second, we would like to introduce some of the interesting examples that have been developed by other authors since those papers were written. Third, we discuss some possible approaches to working in a more flexible context, namely when we have diagrams with both left and right Quillen functors, and why they do not behave as we might wish.

A first question one might ask is why we would want a notion of homotopy limits of model categories in the first place.  Just as in any other context in which we have various diagrams of mathematical objects, when we have Quillen pairs between model categories it is reasonable to ask whether more complicated diagrams of such have a corresponding model category with appropriate universal properties.  Since the purpose of model categories is arguably to work in a good homotopy-theoretic setting, it is only natural that we would want such a construction to be homotopy invariant in some suitable way.

Often when we want to consider homotopy limits of diagrams, we work in an ambient model category, or more general homotopy theory, in which we have built-in methods for constructing them.  However, there is no known ``model category of model categories", so we must content ourselves with ad hoc constructions. Nonetheless, we would like some way to reassure ourselves that the homotopy limit construction that we define has some right to have that name.  We can do so by migrating to a suitable model of $(\infty,1)$-categories in which homotopy limits can be defined unambiguously.  We do not revisit this issue in this paper, but note that we work with the complete Segal space model in \cite{fiberprod} and \cite{holim}; Harpaz has proved a similar result in the quasi-category model \cite{harpaz}.

After an introduction to some necessary model category background, we begin the main part of the paper with the construction of homotopy limits of model categories, beginning with the more tangible special case of homotopy pullbacks.  We then present a sampling of examples for which this construction has proved useful.  Since this paper is being written for the proceedings of John Greenlees' birthday conference, we give particular attention to the example of adelic models for tensor-triangulated categories that he developed with Balchin \cite{bg}.  We refer the reader to \cite{hr} and \cite{hkvww} for related constructions and examples thereof.

At the end of the paper, we discuss the rigidity of this construction.  Specifically, while it can be modified to the context of model categories and right Quillen functors, we do not get homotopy limits with the correct properties if we try to modify our construction to diagrams consisting of both kinds of functors.

In the paper \cite{hocolim}, we consider a similar definition for homotopy colimits of model categories.  However, for a number of reasons that construction is not as satisfactory as this one, so we do not include it here.  We refer the reader to work of Harpaz and Prasma \cite{hp} for another approach to homotopy colimits.

\begin{thank}
	We would like to thank Scott Balchin for discussions about this paper, as well as the referee for many helpful comments.
\end{thank}

\section{Model categories and localizations}

In this section we give a brief review of some model category notions and techniques that we need for our construction and some of our examples.

Recall that a \emph{model category} $\mathcal M$ is a category with three distinguished classes of morphisms: weak equivalences, fibrations, and cofibrations, satisfying five axioms \cite[3.3]{ds}, \cite{quillen}.  An object $x$ in $\mathcal M$ is \emph{fibrant} if the unique map $x \rightarrow \ast$ to the terminal object is a fibration.  Dually, an object $x$ in $\mathcal M$ is \emph{cofibrant} if the unique map $\varnothing \rightarrow x$ from the initial object is a cofibration.

We also want to consider appropriate functors between model categories.  Because of the dual nature of the fibrations and cofibrations, it is sensible to consider adjoint pairs of functors in this situation.

\begin{definition}
	\begin{enumerate}
		\item Let $\mathcal M$ and $\mathcal N$ be model categories.  An adjoint pair of functors
		\[ F \colon \mathcal M \rightleftarrows \mathcal N \colon G \]
		is a \emph{Quillen pair} if the left adjoint $F$ preserves cofibrations and the right adjoint $G$ preserves fibrations.  We refer to $F$ as a \emph{left Quillen functor} and similarly refer to $G$ as a \emph{right Quillen functor}.  
		
		\item A Quillen pair is a \emph{Quillen equivalence} if, for any cofibrant object $x$ of $\mathcal M$ and any fibrant object $y$ of $\mathcal N$, a map $Fx\rightarrow y$ is a weak equivalence in $\mathcal N$ if and only if its adjoint map $x \rightarrow Gy$ is a weak equivalence in $\mathcal M$.
	\end{enumerate}
\end{definition}

We note that an equivalent definition of Quillen pair is that the left adjoint $F$ preserves cofibrations and acyclic cofibrations.  Using this approach, we could consider left Quillen functors independently of having a right adjoint functor, and analogously consider right Quillen functors as those that preserve fibrations and acyclic fibrations. 

\begin{example} \label{ssetsmc}
Recall that the simplicial indexing category $\Deltaop$ is defined to be the category with objects finite ordered sets $[n]=\{0 \rightarrow 1 \rightarrow \cdots \rightarrow n\}$ and morphisms the opposites of the order-preserving maps between them.  A \emph{simplicial set} is a functor
\[ K \colon \Deltaop \rightarrow \Sets. \]
We denote by $\SSets$ the category of simplicial sets, and this category has a natural model structure Quillen equivalent to the standard model structure on topological spaces originally due to Quillen \cite[I.10]{gj}.  A map of simplicial sets is a weak equivalence precisely if its geometric realization is a weak homotopy equivalence of topological spaces.  The cofibrations are the monomorphisms, and the fibrations are the \emph{Kan fibrations}, or maps with the right lifting property with respect to the horn inclusions $\Lambda^k[n] \rightarrow \Delta[n]$ for $k \geq 1$ and $0 \leq k \leq n$.

Although there are many other possible model structures on the category of simplicial sets, throughout this paper we assume this model structure unless otherwise stated.
\end{example}

\begin{example} \label{m1}
	Given a model category $\mathcal M$, there is also a model structure on the category $\mathcal M^{[1]}$ whose objects of $\mathcal M^{[1]}$ are morphisms of $\mathcal M$, and whose morphisms are given by pairs of morphisms in $\mathcal M$ making the appropriate square diagram commute.  A morphism in $\mathcal M^{[1]}$ is a weak equivalence if its component maps are weak equivalences in $\mathcal M$, and cofibrations are defined analogously.  More generally, $\mathcal M^{[n]}$ is the category with objects strings of $n$ composable morphisms in $\mathcal M$; the model structure can be defined analogously.
\end{example}

We now consider some additional features that a model category might have, beginning with the notion of properness.

\begin{definition}
	A model category $\mathcal M$ is \emph{right proper} if every pullback of a weak equivalence along a fibration is a weak equivalence.  It is \emph{left proper} if every pushout of a weak equivalence along a cofibration is a weak equivalence.  It is \emph{proper} if it is both left and right proper.
\end{definition}

\begin{example}
	The model structure for simplicial sets from Example \ref{ssetsmc} is proper.  The fact that it is left proper follows from the fact that all objects are cofibrant, while the fact that it is right proper requires a bit more verification \cite[13.1.13]{hirsch}.
\end{example}

In practice, most of the model structures that we consider have the additional structure of being cofibrantly generated.

\begin{definition} 
	A model category $\mathcal M$ is \emph{cofibrantly generated} if there exist sets $I$ and $J$, called the \emph{generating cofibrations} and \emph{generating acyclic cofibrations}, respectively, each satisfying the small object argument \cite[10.5.15]{hirsch}, such that a map in $\mathcal M$ is a fibration if and only if it has the right lifting property with respect to the maps in $J$, and a map is an acyclic fibration if and only if it has the right lifting property with respect to the maps in $I$.
\end{definition}

\begin{example}
	The model structure on simplicial sets is cofibrantly generated.  The set $I$ of generating cofibrations can be taken to be the set of boundary inclusions
	\[ \{ \partial \Delta[n] \rightarrow \Delta[n] \mid n \geq 0 \}, \]
	and the set $J$ of generating acyclic cofibrations can be taken to be the set of horn inclusions
	\[ \{ \Lambda^k[n] \rightarrow \Delta[n] \mid n \geq 1, 0 \leq k \leq n\}. \]
\end{example}

We now want to look at model categories which are combinatorial, which, loosely speaking, can be considered to be model categories that behave sufficiently like the model structure on simplicial sets.  The formal definition requires some technical points to be developed, however. Since, in particular, they require the use of filtered colimits, we begin by recalling the definition of filtered diagrams, over which such colimits are taken.

\begin{definition}
	Let $\lambda$ be a regular cardinal.  
	\begin{enumerate}
		\item A category $\mathcal I$ is $\lambda$-\emph{filtered} if any subcategory with fewer than $\lambda$ morphisms has a compatible cocone in $\mathcal I$.
		
		\item An object $c$ of a category $\mathcal C$ is $\lambda$-\emph{small} if for any $\lambda$-filtered category $\mathcal I$ and diagram $X \colon \mathcal I \rightarrow \mathcal C$, 
		\[ \Hom(c, \colim X_i) \cong \colim \Hom(c, X_i). \]
	\end{enumerate}
\end{definition}

\begin{definition} \cite[2.2]{duggercomb}
	A category $\mathcal C$ is \emph{locally presentable} if it is closed under colimits and if there is a regular cardinal $\lambda$ and a set of objects $A$ in $\mathcal C$ such that:
	\begin{enumerate}
		\item every object in $A$ is small with respect to $\lambda$-filtered colimits, and
		
		\item every object in $\mathcal C$ can be expressed as a $\lambda$-filtered colimit of elements of $A$.
	\end{enumerate}
\end{definition}

In what follows, we use the terminology ``filtered colimits" for $\lambda$-filtered colimits for some sufficiently large regular cardinal $\lambda$.

We can now state the definition of a combinatorial model category.

\begin{definition} \cite[2.1]{duggercomb} \label{combmc}
	A model category is	\emph{combinatorial} if it is cofibrantly generated and locally presentable.
\end{definition}

For the purposes of this paper, we want to consider combinatorial model categories because they can often be localized with respect to sets of maps to obtain a new model structure on the same underlying category but with more weak equivalences.  To formalize this idea, we need to consider mapping spaces.

It is often helpful, when working with a model category, to consider not only sets of morphisms between two given objects, but simplicial sets of morphisms. 

\begin{definition}
	A \emph{simplicial category} or \emph{category enriched in simplicial sets} is a category $\mathcal C$ such that for any objects $x$ and $y$ of $\mathcal C$, there is a simplicial set $\Map_\mathcal C(x,y)$ of morphisms from $x$ to $y$, together with a composition law satisfying associativity and unitality.  These simplicial sets are called \emph{mapping spaces}.
\end{definition}

\begin{example}
	The category of simplicial sets has the structure of a simplicial category, where, given simplicial sets $K$ and $L$, the mapping space $\Map(K,L)$ is defined to have $n$-simplices
	\[ \Map(K,L)_n = \Hom(K \times \Delta[n], L), \]
	where $\Delta[n]$ is the standard $n$-simplex.
\end{example} 

Any model category can be equipped with such mapping spaces, chosen in a homotopy invariant way, even if the underlying category does not have a canonical simplicial category structure.  There are a number of ways to define such \emph{homotopy mapping spaces}, including the simplicial localization constructions of Dwyer and Kan \cite{dkfncxes}, \cite{dksimploc}, or the approach via framings of Hirschhorn \cite[\S 19.1]{hirsch}.  We denote the homotopy mapping space from an object $x$ to any object $y$ by $\Map^h(x,y)$.

We can now set up the necessary definitions to discuss localizations of model categories.

\begin{definition}
	Let $\mathcal M$ be a model category and $S$ a class of maps in $\mathcal M$.  
	\begin{enumerate}
	\item An object $x$ of $\mathcal M$ is $S$-\emph{colocal} if it is cofibrant and if for every map $f \colon a \rightarrow b$ in $\mathcal S$, the induced map of homotopy mapping spaces 
	\[ \Map^h_\mathcal M(x,a) \rightarrow \Map^h_\mathcal M(x,b) \] 
	is a weak equivalence of simplicial sets.
	
	\item A map $f \colon d \rightarrow d$ is an $S$-\emph{colocal equivalence} if, for every $S$-colocal object $x$, the induced map of homotopy mapping spaces 
	\[ \Map^h_\mathcal M(x,c) \rightarrow \Map^h_\mathcal M(x,d) \]
	is a weak equivalence of simplicial sets.
	\end{enumerate}
\end{definition}

We often consider the following notion of colocal objects with respect to a set of objects.

\begin{definition} 
	Let $\mathcal M$ be a model category and $K$ a set of objects in $\mathcal M$.  
	A map $x \rightarrow y$ is a $K$-\emph{local equivalence} if, for every object $a$ in $K$, the induced map of homotopy mapping spaces
	\[ \Map^h(a,x) \rightarrow \Map^h(a,y) \]
	is a weak equivalence of simplicial sets.
\end{definition}

\begin{definition}
	Let $\mathcal M$ be a model category, $K$ a set of objects in $\mathcal M$, and $S$ the class of $K$-colocal equivalences.  The \emph{right Bousfield localization} of $\mathcal M$ with respect to $S$, if it exists, is a model structure $\mathcal R_S\mathcal M$ on the underlying category of $\mathcal M$ such that:
	\begin{enumerate}
		\item the weak equivalences are the $S$-colocal equivalences, and 
		
		\item the fibrations are precisely the fibrations of $\mathcal M$.
	\end{enumerate}
\end{definition}

Observe that we could consider such a localization as being with respect to the set $K$ of objects, rather than with respect to the class $S$ of maps.  We have chosen to use the notation $S$ to be parallel with the definition of left Bousfield localization below.

Finally, we bring together the different flavors of model categories we have considered to state the following existence theorem for right Bousfield localizations of model categories.

\begin{theorem} \cite[5.1.1]{hirsch}, \cite[\S 5]{barwick} \label{rightbousexist}
	Let $\mathcal M$ be a right proper combinatorial model category, $K$ a set of objects in $\mathcal M$, and $S$ the class of $K$-colocal equivalences.
	\begin{enumerate}
		\item The right Bousfield localization $\mathcal R_S\mathcal M$ of $\mathcal M$ exists.
		
		\item The cofibrant objects of $\mathcal R_S\mathcal M$ are precisely the $K$-colocal objects of $\mathcal M$.
		
		\item The model category $\mathcal R_S\mathcal M$ is right proper.
	\end{enumerate}
\end{theorem}

Much more common in the literature is the notion of left Bousfield localization, which we now review.  It is not needed for our main construction, but does appear in some of the applications we discuss.

\begin{definition}
	Let $\mathcal M$ be a model category and $S$ a set of maps in $\mathcal M$.
	\begin{enumerate}
		\item A fibrant object $z$ of $\mathcal M$ is $S$-\emph{local} if for every map $a \rightarrow b$ in $S$, the induced map of simplicial sets
		\[ \Map^h(b,z) \rightarrow \Map^h(a,z) \]
		is a weak equivalence.
		
		\item A morphism $x \rightarrow y$ in $\mathcal M$ is an $S$-\emph{local equivalence} if for any $S$-local object $Z$, the induced map
		\[ \Map^h(y, z) \rightarrow \Map^h(x,z) \]
		is a weak equivalence of simplicial sets.
	\end{enumerate}
\end{definition}

\begin{definition}
	Let $\mathcal M$ be a model category and $S$ a set of maps in $\mathcal M$.  The \emph{left Bousfield localization} of $\mathcal M$ with respect to $S$, if it exists, is a model structure $\mathcal L_S \mathcal M$ on the same underlying category as $\mathcal M$, such that:
	\begin{enumerate}
		\item the weak equivalences are the $S$-local equivalences, and
		
		\item the cofibrations are exactly the cofibrations of $\mathcal M$.
	\end{enumerate}
\end{definition}

\begin{theorem} \cite[4.1.1]{hirsch}, \cite[\S 4]{barwick}
	Let $\mathcal M$ be a left proper combinatorial model category and $S$ a set of maps in $\mathcal M$.
	\begin{enumerate}
		\item The left Bousfield localization $\mathcal L_S\mathcal M$ of $\mathcal M$ exists.
		
		\item The model category $\mathcal L_S\mathcal M$ is left proper.
	\end{enumerate}
\end{theorem}		

\section{Homotopy pullbacks of model categories}

In this section, we recall the construction of the homotopy pullback of an appropriate diagram of model categories, and we discuss when it can be given the structure of a model category itself.

The following definition first appeared in the literature in \cite{toendha}.  There, and in our original paper \cite{fiberprod}, it was called a \emph{homotopy fiber product} instead.  Here, we distinguish more explicitly between the homotopy limit and a more general preliminary category of diagrams.

\begin{definition}
Suppose that
\[ \xymatrix@1{\mathcal M_1 \ar[r]^{F_1} & \mathcal M_3 & \mathcal M_2 \ar[l]_{F_2}} \] 
is a diagram of left Quillen functors of model categories.  Define their \emph{lax homotopy pullback} to be the category $\mathcal M$ whose objects are given by 5-tuples $(x_1, x_2, x_3; u, v)$ such that each $x_i$ is an object of $\mathcal M_i$ fitting into a diagram
\[ \xymatrix@1{F_1(x_1) \ar[r]^-u & x_3 & F_2(x_2) \ar[l]_-v.} \]
A morphism of $\mathcal M$, say $f \colon (x_1, x_2, x_3; u, v) \rightarrow (y_1, y_2, y_3; z, w)$, is given by a triple of maps $f_i \colon x_i \rightarrow y_i$ for $i=1,2,3$, such that the diagram 
\[ \xymatrix{F_1(x_1) \ar[r]^-u \ar[d]^{F_1(f_1)} & x_3 \ar[d]^{f_3} & F_2(x_2) \ar[l]_-v \ar[d]^{F_2(f_2)} \\
F_1(y_1) \ar[r]^-z & y_3 & F_2(y_2) \ar[l]_-w} \]
commutes.
\end{definition}

\begin{remark}
	For ease of notation, we have simply denoted this category by $\mathcal M$, rather than something more suggestive of a diagram category or even of a homotopy pullback.  In the later section on homotopy limits, we give more explicit notation that one could also employ here if more clarity is needed.
\end{remark}

The lax homotopy pullback $\mathcal M$ can be given the structure of a model category, where the weak equivalences and cofibrations are given levelwise.  In other words, $f$ is a weak equivalence if each map $f_i$ is a weak equivalence in $\mathcal M_i$, and cofibrations are defined analogously.

However, the model category $\mathcal M$ is not yet what we want.  We would like to require the maps $u$ and $v$ to be weak equivalences rather than arbitrary maps.

\begin{definition}
	Suppose that
	\[ \xymatrix@1{\mathcal M_1 \ar[r]^{F_1} & \mathcal M_3 & \mathcal M_2 \ar[l]_{F_2}} \] 
	is a diagram of left Quillen functors of model categories.  Its \emph{homotopy pullback} is the full subcategory of the lax homotopy pullback for which the maps $u$ and $v$ are weak equivalences.
\end{definition}

However, this restriction comes at a price. On its own, the homotopy pullback cannot be given the structure of a model category because it is not generally closed under limits and colimits.  The typical solution in this kind of situation is rather to find a localization of the model structure on the lax homotopy pullback $\mathcal M$ so that the fibrant and cofibrant objects have the maps $u$ and $v$ weak equivalences.  In that direction, we have the following theorem.

\begin{theorem} \cite[3.1]{fiberprod} \label{limmc}
	Let $\mathcal M$ be the lax homotopy pullback of a diagram of left Quillen functors
	\[ \xymatrix@1{\mathcal M_1 \ar[r]^{F_1} & \mathcal M_3 & \mathcal M_2 \ar[l]_{F_2} } \] 
	where each of the categories $\mathcal M_i$ is combinatorial.  Further assume that $\mathcal M$ is right proper. Then there exists a right Bousfield localization of $\mathcal M$ whose cofibrant objects $(x_1, x_2, y_2; u, v)$ have both $u$ and $v$ weak equivalences in $\mathcal M_3$.
\end{theorem}

\begin{proof}
	Since we have assumed that the categories $\mathcal M_1$, $\mathcal M_2$, and $\mathcal M_3$ are combinatorial, and in particular locally presentable, we can find, for each $i=1,2,3$, a set $\mathcal A_i$ of objects of $\mathcal M_i$ generating all of $\mathcal M_i$ by filtered colimits.  We further assume that the objects of each $\mathcal A_i$ are all cofibrant in the corresponding model category $\mathcal M_i$.  For an explicit construction for such sets, we refer the reader to presentations of combinatorial model categories as given by Dugger \cite{duggercomb}.  
	
	Given $a_1 \in \mathcal A_1$ and $a_2 \in \mathcal A_2$, consider the class of all objects $x_3$ in $\mathcal M_3$ such that there are pairs of weak equivalences
	\[ \xymatrix@1{F_1(a_1) \ar[r]^-\simeq  & x_3 & F_2(a_2). \ar[l]_-\simeq} \]  
	For every choice of $a_1$ and $a_2$ such that $F_1(a_1)$ is weakly equivalent to $F_2(a_2)$, choose a single cofibrant object $x_3$ in this weak equivalence class of objects; such a choice is possible since $\mathcal A_1$ and $\mathcal A_2$ have been chosen to be sets.  Define the set $\mathcal B_3$ to be the union of this set of all such objects $x_3$ together with the generating set $\mathcal A_3$ for $\mathcal M_3$.  To retain uniformity of notation going forward, for $i=1,2$, let $\mathcal B_i=\mathcal A_i$.
	
	Consider the set 
	\[ \{(x_1, x_2, x_3;u,v) \mid x_i \in \mathcal B_i, u, v \text{ weak equivalences in } \mathcal M_3\} \] 
	of objects of the lax homotopy limit $\mathcal M$. 
	Taking filtered colimits of such objects, we can obtain the set $\mathcal B$ of all objects $(x_1, x_2, x_3;u,v)$ of $\mathcal M$ for which the maps $u$ and $v$ are weak equivalences.  Although, as we discussed earlier, arbitrary colimits do not necessarily preserve weak equivalences, an important feature of filtered colimits is that they do \cite[7.3]{duggercomb}. 
	
	Now we take a right Bousfield localization of $\mathcal M$ with respect to the set $\mathcal B$.  Since $\mathcal M$ is given by a diagram of combinatorial model categories, it is itself combinatorial, so if $\mathcal M$ is right proper, then this localization has a model structure by Theorem \ref{rightbousexist}.  The class of cofibrant objects of this model category is precisely the smallest class of cofibrant objects of $\mathcal M$ containing this set and closed under homotopy colimits and weak equivalences \cite[5.1.5, 5.1.6]{hirsch}.  Thus, our only remaining step is to show that homotopy colimits of objects in $\mathcal B$ still have $u$ and $v$ weak equivalences.
	
	Suppose that $\mathcal C$ is a small category and $X \colon \mathcal C \rightarrow \mathcal M$ a functor such that the objects in the image of $\mathcal C$ are in the set $\mathcal B$.  In other words, for any object $\alpha$ of $\mathcal C$, we have
	\[ X(\alpha)=(x_1^\alpha, x_2^\alpha, x_3^\alpha; u^\alpha, v^\alpha) \] 
	with $x_i^\alpha \in \mathcal B_i$ and the maps $u^\alpha$ and $v^\alpha$ weak equivalences.  Notice that for each $i=1,2,3$, we have diagrams $X_i \colon \mathcal C \rightarrow \mathcal M_i$ such that for each $\alpha$ in $\mathcal C$, $X_i(\mathcal \alpha) = x_i^\alpha$.  Thus, the homotopy colimit of the original diagram $X$ is the object
	\[ (\hocolim X_1, \hocolim X_2, \hocolim X_3; \hocolim (u), \hocolim (v)) \] 
	in $\mathcal M$.  Since we have assumed that each object $x_i^\alpha$ is in $\mathcal B_i$ and hence cofibrant, we get that the maps $\hocolim (u)$ and $\hocolim (v)$ are weak equivalences in $\mathcal M_3$ by \cite[19.4.2]{hirsch}.
\end{proof}

Unfortunately, the right properness assumption in this theorem is quite restrictive.  In particular, we have not been able to identify conditions on the model categories $\mathcal M_1$, $\mathcal M_2$, and $\mathcal M_3$ guaranteeing that the model category $\mathcal M$ is right proper.  An option for weakening the structure of a model category to accommodate localizations of non-right proper model categories is given by Barwick \cite{barwick}.  Alternatively, when the conditions of this theorem are not satisfied, we can still use the levelwise model structure on the lax homotopy pullback $\mathcal M$ and simply restrict to the homotopy pullback as a subcategory when necessary.

\begin{remark} \label{whyleft}
One might ask why we have chosen to use left Quillen functors.  We could make a dual definition in which we instead use right Quillen functors; the arrows $u$ and $v$ in the definition of the homotopy limit would then more naturally be taken to point in the opposite direction, and underlying model structure should then be taken to be the projective model structure, in which the weak equivalences and fibrations are taken to be levelwise.  Being careful about such details, it is not hard to verify that this dual construction works perfectly well.

One might then wonder if one could find a sensible definition of a homotopy pullback of a left Quillen functor along a right Quillen functor.  Because this discussion requires other homotopy limits of model categories, we defer it to Section \ref{mixed}, after we have discussed the definition of general homotopy limits.
\end{remark}

\section{Examples of homotopy pullbacks of model categories}

Let us now consider some examples of this construction.   We start with some basic general examples that we also described in \cite{fiberprod}.

\begin{example} \label{fiber}
	Let us consider the following special case of a homotopy pullback, the homotopy fiber of a map.  
	
	Let $F \colon \mathcal M \rightarrow \mathcal N$ be a left Quillen functor of model categories.  Then the \emph{homotopy fiber} of $F$ is the homotopy fiber product of the diagram
	\[ \xymatrix{ & \mathcal M \ar[d]^F \\
		\ast \ar[r] & \mathcal N} \] 
	where the map $\ast \rightarrow \mathcal N$ is necessarily the map from the trivial model category on a single object $\ast$ to the initial object $\varnothing$ of $\mathcal N$.
	
	Using the definition of the homotopy pullback, the objects of the homotopy fiber are triples $(\ast, m, n; u, v)$, $m$ is an object of $\mathcal M$, $n$ is an object of $\mathcal N$, $u \colon \varnothing \rightarrow n$ is the unique such map, and $v \colon F(m) \rightarrow n$.  Imposing our condition that $u$ and $v$ be weak equivalences, we get that $n$ must be weakly equivalent to the initial object of $\mathcal N$, and $m$ is any object of $\mathcal M$ whose image under $F$ is weakly equivalent to the initial object of $\mathcal N$.
	
	Here, we see that the requirement that functors be left Quillen is prohibitively restrictive for many examples one might want to consider!  Although we would like to look at the homotopy fiber over arbitrary objects, it is not possible under this hypothesis.

	A further specialization illustrates its particularly odd nature still more.  If we take the analogue of a loop space and define the ``loop model category" as the homotopy pullback of the diagram
	\[ \xymatrix{ & \ast \ar[d] \\
		\ast \ar[r] & \mathcal M} \] 
	for any model category $\mathcal M$, we simply get the subcategory of $\mathcal M$ whose objects are weakly equivalent to the initial object. 
\end{example}

\begin{example}
	To\"en's motivation for defining the homotopy pullback of model categories was to prove associativity of his derived Hall algebras \cite{toendha}.  In this context, we work with a stable model category $\mathcal N$; this extra assumption that the homotopy category is triangulated implies that $\mathcal N$ has a zero object 0 that is both initial and terminal.
	
	Recall the model structure on the morphism category $\mathcal N^{[1]}$ from Example \ref{m1}.  Given an object of $\mathcal N^{[1]}$, namely a map $f \colon x \rightarrow y$ in $\mathcal N$, let $F \colon \mathcal N^{[1]} \rightarrow \mathcal N$ be the target map, so that $F(f \colon x \rightarrow y)= y$.  Let $C \colon \mathcal N^{[1]} \rightarrow \mathcal N$ be the cone map, so that $C(f\colon x \rightarrow y)= y \amalg_x 0$.  Using these functors, we get a diagram
	\[ \xymatrix{ & \mathcal N^{[1]} \ar[d]^C \\
		\mathcal N^{[1]} \ar[r]^F & \mathcal N.} \]
	The homotopy pullback of this diagram is Quillen equivalent to the model category $\mathcal N^{[2]}$ whose objects are pairs of composable morphisms in $\mathcal N$ \cite[\S 4]{toendha}.  The idea is that an object of the homotopy pullback consists of 
	\[ \left(w \rightarrow z, x \rightarrow y, z', u, v \right), \]
	where 
	\[ \xymatrix@1{F(w \rightarrow z) \ar[r]^-u_-\simeq & z' & C(x \rightarrow y). \ar[l]_-v^-\simeq } \]
	Applying the definitions of $F$ and $C$, we get
	\[ \xymatrix@1{z \ar[r]^\simeq & z' & y \amalg_x 0. \ar[l]_-\simeq } \]
	This data is essentially that of a distinguished triangle $x \rightarrow y \rightarrow z$ in the homotopy category, which is precisely an object of $\mathcal N^{[2]}$. 
\end{example}

\begin{example}
	The next example lifts a result in stable homotopy theory to the level of model categories, as first proved by Guti\'errez and Roitzheim \cite{gr}.
	
	Given a prime $p$, let $\mathbb Z_p$ denote the $p$-adic integers.  Then the integers are given by the pullback
	\[ \xymatrix{\mathbb Z \ar[r] \ar[d] & \prod_p \mathbb Z_p \ar[d] \\
	\mathbb Q \ar[r] & \mathbb Q \otimes_\mathbb Z \left( \prod_p \mathbb Z_p \right).} \]
	This ``arithmetic fracture square" has further generalizations in number theory, but in homotopy theory has the following upgrade to the context of spectra.
	
	Let $E$ be a spectrum, and consider the localization functor $L_E$ that inverts maps of spectra that are isomorphisms after applying $E_*$-homology.  Furthermore, given any abelian group $G$, let $MG$ denote its corresponding Moore spectrum, namely, the connective spectrum whose 0th homology group is $G$ and whose other homotopy groups are trivial \cite[III.6]{adams}.  Suppose that $J$ and $K$ form a partition of the set of all prime numbers, and let $\mathbb Z_J$ and $\mathbb Z_K$ denote the set of integers localized at the sets $J$ and $K$, respectively.  Then a classical result is that any spectrum $X$ is the homotopy pullback of the diagram of spectra
	\[ \xymatrix{X \ar[r] \ar[d] & L_{M\mathbb Z_J}X \ar[d] \\
	L_{M\mathbb Z_K}X \ar[r] & L_{M\mathbb Q} X.} \]

	We now use the homotopy pullback of model categories to give a further upgrade to the level of model categories.  Let $\Sp$ denote the model structure for symmetric spectra.  As above, let $E$ be any spectrum, and let $L_E\Sp$ denote the left Bousfield localization of $\Sp$ with respect to the $E_*$-equivalences.  As before, we consider the Moore spectra $M\mathbb Z_J$ and $M\mathbb Z_K$ associated to some partition $J,K$ of the set of primes.  Guti\'errez and Roitzheim prove that $\Sp$ is Quillen equivalent to the homotopy limit of the diagram of left Quillen functors
	\[ \xymatrix{\Sp \ar[r] \ar[d] & L_{M\mathbb Z_J} \Sp \ar[d] \\
		L_{M\mathbb Z_K} \Sp \ar[r] & L_{M\mathbb Q} \Sp.} \]
	Indeed, they prove that this result holds for any left proper, combinatorial, stable model category $\mathcal C$ in place of the model structure $\Sp$ \cite[3.4]{gr}.  See also \cite{duggersp} for details about how why such a localization is possible.
\end{example}



\section{Homotopy limits of model categories}

Let us now discuss how to generalize the homotopy pullback construction to more general homotopy limits of model categories.  While the notation gets more complicated, the essential idea is the same.

If $\mathcal D$ is a small category, then a $\mathcal D$-shaped diagram $\mathcal M$ is given by a collection $\mathcal M_\alpha$, one for each object $\alpha$ of $\mathcal D$, together with left Quillen functors $F_{\alpha, \beta}^\theta \colon \mathcal M_\alpha \rightarrow \mathcal M_\beta$, compatible with one another in the sense that, if $\theta \colon \alpha \rightarrow \beta$ and $\delta \colon \beta \rightarrow \gamma$ are composable maps in $\mathcal D$, then
\[ F_{\alpha, \gamma}^{\delta \theta}=F_{\beta, \gamma}^\delta \circ F_{\alpha, \beta}^\theta. \]
The superscript $\theta$ indexes different left Quillen functors between the same model categories, corresponding to distinct maps $\theta \colon \alpha \rightarrow \beta$.  More precisely, if we consider the (large) category $\mathcal {MC}$ of model categories with left Quillen functors between them, or some small subcategory of it, then we can describe such a diagram as by a functor $\mathcal M \colon \mathcal D \rightarrow \mathcal {MC}$.

\begin{definition}
Let $\mathcal M$ be a $\mathcal D$-shaped diagram of model categories.  Then the \emph{lax homotopy limit} of $\mathcal M$, denoted by $\mathcal L_\mathcal D \mathcal M$, has objects families $(x_\alpha, u^\theta_{\alpha, \beta})$ where $x_\alpha$ is an object of $\mathcal M_\alpha$ and $u^\theta_{\alpha, \beta} \colon F^\theta_{\alpha, \beta}(x_\alpha) \rightarrow x_\beta$ is a morphism in $\mathcal M_\beta$, satisfying the compatibility condition
\[ u_{\alpha, \gamma}^{\delta \theta}= u_{\beta, \gamma}^\delta \circ F_{\beta, \gamma}^\delta(u_{\alpha, \beta}^\theta). \]
Morphisms are component-wise maps of such families, making the appropriate diagrams commute.

The \emph{homotopy limit} for $\mathcal M$, denoted by $\mathcal Lim_\mathcal D \mathcal M$, is the full subcategory of $\mathcal L_\mathcal D \mathcal M$ whose objects satisfy the additional condition that all maps $u_{\alpha, \beta}^\theta$ are weak equivalences in their respective $\mathcal M_\beta$.
\end{definition}

As for the lax homotopy pullback, the lax homotopy limit $\mathcal L_\mathcal D \mathcal M$ can be given the injective model structure, assuming that all model categories in the diagram are sufficiently nice.  On the other hand, just as in the case of the homotopy pullback, $\mathcal Lim_\mathcal D \mathcal M$ does not have the structure of a model category, since the weak equivalence requirement is not preserved by general limits and colimits.  Again, we can ask whether we can find a localization of the more general model structure so that the fibrant-cofibrant objects do have the maps weak equivalences, and obtain the following result.

\begin{theorem} \cite[3.2]{holim}
Let $\mathcal L_\mathcal D \mathcal M$ be the lax homotopy limit of a $\mathcal D$-diagram combinatorial model categories $\mathcal M_\alpha$, and assume that $\mathcal L_\mathcal D \mathcal M$ has the structure of a right proper model category.  Then there exists a right Bousfield localization of $\mathcal L_\mathcal D \mathcal M$ whose cofibrant objects $(x_\alpha, u_{\alpha, \beta}^\theta)$ have all $x_\alpha$ cofibrant and all maps $u_{\alpha, \beta}^\theta$ weak equivalences in $\mathcal M_\beta$.
\end{theorem}

The proof is very similar to the one for homotopy pullbacks, so we omit it here, referring the reader to \cite{holim} for the details.  The concern about whether a given diagram actually satisfies these conditions is also analogous to the situation for the homotopy pullback, as are the potential resolutions.

\section{Examples of homotopy limits of model categories}

We now consider some examples of homotopy limits of model categories, the first few of which were given in \cite{holim}.

\begin{example}
Let $\mathcal D$ be the category $\bullet \rightarrow \bullet$.  Then, a corresponding diagram of model categories has the form $F \colon \mathcal M_1 \rightarrow \mathcal M_2$.  The homotopy limit has objects $(x_1, x_2; u)$ where $x_i$ is an object of $\mathcal M_i$ for $i=1,2$, and $u \colon F(x_1) \rightarrow x_2$ is a weak equivalence in $\mathcal M_2$.  Thus, the homotopy limit is given by the weak essential image of $\mathcal M_1$ in $\mathcal M_2$, i.e., the subcategory of $\mathcal M_2$ whose objects are weakly equivalent to objects in the image of $F$.
\end{example}

\begin{example}
Let us now consider an equalizer diagram $\bullet \rightrightarrows \bullet$.  Such a diagram of model categories looks like
\[ \xymatrix@1{\mathcal M_1 \ar@<.5ex>[r]^{F_1} \ar@<-.5ex>[r]_{F_2} & \mathcal M_2.} \]  
The homotopy limit has objects $(x_1, x_2; u_1, u_2)$ with $x_i$ an object of $\mathcal M_i$ for $i=1,2$ and $u_1, u_2$ weak equivalences
\[ \xymatrix@1{F_1(x_1) \ar[r]^-{u_1} & x_2 & F_2(x_1). \ar[l]_-{u_2}} \]  
Thus, the equalizer of model categories is equivalent to the subcategory of $\mathcal M_2$ which is in the weak essential image of both $F_1$ and $F_2$.
\end{example}

\begin{example}
As a more interesting example, we show that the category of simplicial sets is equivalent to the homotopy limit of model categories of $n$-\emph{types}, or topological spaces with nontrivial homotopy groups concentrated in degrees $n$ and below.  We apply the homotopy limit construction in the context of simplicial sets, so that we have a combinatorial model structure, but can invoke the Quillen equivalence with topological spaces throughout to recover statements about topological spaces.

Consider the model category $\SSets$ of topological spaces.  For each natural number $n$, take a left Bousfield localization of $\SSets$, denoted by $\SSets_{\leq n}$, with respect to the set of maps
\[ \{\partial \Delta[k] \rightarrow \Delta[k] \mid k>n+1\}. \]  
The weak equivalences in $\SSets_{\leq n}$ are the $n$-\emph{equivalences}, or maps $X \rightarrow Y$ such that the induced maps
\[ \pi_i(|X|) \rightarrow \pi_i(|Y|) \] 
are weak homotopy equivalences for all $i \leq n$.  These model structures form a diagram of left Quillen functors
\[ \cdots \rightarrow \SSets_{\leq 3} \rightarrow \SSets_{\leq 2} \rightarrow \SSets_{\leq 1} \rightarrow \SSets_{\leq 0}. \]

The homotopy limit of this diagram has objects
\[ \cdots \rightarrow X_3 \rightarrow X_2 \rightarrow X_1 \rightarrow X_0 \] 
for which each map $X_{n+1} \rightarrow X_n$ is an $n$-equivalence, and morphisms given by $n$-equivalences $X_n \rightarrow Y_n$ for all $n \geq 0$ such that the resulting diagram commutes.

There is a functor $\SSets \rightarrow \mathcal Lim_n \SSets_{\leq n}$ taking a space $X$ to the constant sequence on $X$.  In the other direction, there is a functor sending a diagram of spaces
\[ \cdots \rightarrow X_3 \rightarrow X_2 \rightarrow X_1 \rightarrow X_0 \] to its homotopy limit $\holim_n X_n$ in $\SSets$.

Taking the constant diagram on a simplicial set $X$, followed by its homotopy limit, recovers $X$.  Composing in the other direction, we need to prove that a diagram of simplicial sets
\[ \cdots \rightarrow X_3 \rightarrow X_2 \rightarrow X_1 \rightarrow X_0 \] 
is equivalent to the constant diagram given by $\holim_n X_n$.  In other words, we want to show that the homotopy limit of this diagram is $n$-equivalent to $X_n$ for all $n \geq 0$, a fact which follows from \cite[19.6.13]{hirsch}. 
\end{example}

The following example is very similar in nature and gives a variant of the Chromatic Convergence Theorem \cite[7.5.7]{rav} at the model category level.

\begin{example}
	For a fixed prime $p$ and finite spectrum $X$, let $L_{n}X$ denote the $n$-th chromatic localization of the spectrum $X$, namely, the localization with respect to the spectrum $E(n)$.  The Chromatic Convergence Theorem states that $X$ is equivalent to the homotopy limit of these localized spectra.
	
	Now, let $\Sp$ denote the model category of symmetric spectra, and $L_n\Sp$ the left Bousfield localization of $\Sp$ with respect to the $E(n)$-equivalences.  After restricting to sufficiently finitary diagrams of spectra (namely, those whose homotopy limits are finite spectra), there is a Quillen pair between $\Sp$ and the homotopy limit of the model categories $L_n\Sp$.  Unfortunately, although Guti\'errez and Roitzheim are able to prove that the unit of this adjunction is a weak equivalence, they are not able to verify what we would hope to be true, that this adjunction is in fact a Quillen equivalence \cite [\S 2.2]{gr}.  
\end{example}

\section{Adelic models of tensor-triangulated categories}

In this section we summarize the role of the homotopy limit of model categories used by Balchin and Greenlees in \cite{bg}.  We have made this example its own section due to the background content necessary to explain it.

The setting for this work is that of tensor-triangulated categories, and more generally, model categories whose homotopy categories have such a structure.  In brief, a tensor-triangulated category is a triangulated category that also has the structure of a monoidal category.  We refer to Balmer's paper for more detail \cite{balmer}.

Let us recall some notions from the general theory.  Recall that a subcategory of a tensor-triangulated category is \emph{thick} if it is closed under retracts and completion of distinguished triangles.  It is an \emph{ideal} if it is closed under completing triangles and tensoring with arbitrary elements of the ambient category.  

We can now define the Balmer spectrum, which is one of the key structures of interest in this section.

\begin{definition} 
	Let $\mathcal T$ be a tensor-triangulated category.
	\begin{enumerate}
		\item A \emph{prime ideal} in $\mathcal T$ is a proper thick ideal $\mathfrak p$ such that if $a \otimes b$ is an object of $\mathfrak p$, then either $a$ or $b$ is also an object of $\mathfrak p$.
		
		\item The \emph{Balmer spectrum} of $\mathcal T$ is the set of prime ideals of the triangulated subcategory of small objects in $\mathcal T$, denoted by $\Spc^\omega(\mathcal T)$.
	\end{enumerate}
\end{definition}

Here, we follow the naming convention of Balchin and Greenlees, in calling ``small" objects that are often referred to as ``compact" in the context of triangulated categories.

In fact, the Balmer spectrum has the structure of a topological space, using the Zariski topology, and as such can be assigned a dimension.  It is \emph{Noetherian} if its open sets satisfy the ascending chain condition.


%

We can now define our primary tensor-triangulated categories of interest.

\begin{definition}
	A tensor-triangulated category is \emph{finite-dimensional Noetherian} if it is rigidly small-generated and its Balmer spectrum is finite-dimensional and Noetherian.
\end{definition}

We have not defined what it means to be rigidly small-generated here, since it is rather technical, but instead refer the reader to \cite[\S 2.A]{bg} for more information.  We shortly discuss in more depth the model category analogue that is of more interest to us here.

In that direction, as we stated at the beginning of this section, we are interested in model categories whose homotopy categories are the kinds of tensor-triangulated categories that we have been discussing and hence have well-behaved Balmer spectra.  We begin with some additional structures on a model category.  The first uses some possibly unfamiliar additional terms, for which we refer the reader to \cite[\S 4.B]{bg}.

\begin{definition}
	A model category $\mathcal C$ is \emph{rigidly compactly generated} if it is a stable, proper, compactly generated symmetric monoidal model category such that:
	\begin{itemize}
		\item the monoidal unit $\mathbb 1$ is $\omega$-cell compact and cofibrant, where $\omega$ denotes the smallest infinite ordinal, and
		
		\item there is a set of generating cofibrations of the form $S^n \otimes \mathcal G \rightarrow D^{n+1} \otimes \mathcal G$ where $\mathcal G$ is a set of $\omega$-cell compact and $\omega$-small cofibrant objects whose images in the homotopy category of $\mathcal C$ are rigid.
	\end{itemize}
\end{definition}

\begin{definition}
	A model category $\mathcal C$ is \emph{finite dimensional Noetherian} if it is a rigidly compactly generated symmetric monoidal model category such that the Balmer spectrum of its homotopy category is Noetherian and finite-dimensional.  In particular, the homotopy category is a finite dimensional Noetherian tensor-triangulated category.
\end{definition}

For the remainder of this section assume that $\mathcal C$ satisfies the conditions of this definition.

Consider the Balmer spectrum of $\mathcal C$.  Given an object $M$ of $\mathcal C$ and Balmer prime $\mathfrak p$, we denote by $L_\mathfrak pM$ the localization of the object $M$ with respect to $\mathfrak p$.  We denote the completion of $M$ at $\mathfrak p$ by $\Lambda_\mathfrak pM$.  We refer the reader to \cite{bg} for a precise description of how to generalize these standard notions from commutative algebra to the particular model category setting considered here.

Let $n$ be the dimension of the Balmer spectrum of the homotopy category of $\mathcal C$ where $\mathcal C$ is as above.  Let $[n] =\{0 \leq 1 \leq \ldots \leq n\}$, and consider the cubical diagram $\mathcal P[n]$ given by the subsets of $[n]$.  Further consider the diagram of nonempty subsets $\mathcal P_0[n]$.

\begin{definition}
	Given any object $M$ of $\mathcal C$, the \emph{adelic diagram} 
	\[ M_{\ad} \colon \mathcal P_0([n]) \rightarrow \mathcal C \]
	is defined by, for any subset $(d_0 > \cdots >d_s)$ of $[n]$, 
	\[ M_{\ad}(d_0 > \cdots >d_s) := \prod_{\dim(\mathfrak p_0)=d_0} L_{\mathfrak p_0} \prod_{\dim(\mathfrak p_1)=d_1, \mathfrak p_1 \subsetneq \mathfrak p_0} L_{\mathfrak p_1} \cdots \prod_{\dim(\mathfrak p_s)=d_s, \mathfrak p_s \subsetneq \mathfrak p_{s-1}} L_{\mathfrak p_s} \Lambda_{\mathfrak p_s} M. \]
	The morphisms are as described in \cite{bg}.
\end{definition}

For example, when $n=1$, we get the homotopy pullback diagram
\[ \xymatrix{& \prod_{\dim(\mathfrak p_0)=0} L_{\mathfrak p_0} \Lambda_{\mathfrak p_0} M \ar[d] \\
	\prod_{\dim(\mathfrak p_1)=1} L_{\mathfrak p_1} \Lambda_{\mathfrak p_1} M \ar[r] & \prod_{\dim(\mathfrak p_0)=0} L_{\mathfrak p_0} \prod_{\dim(\mathfrak p_1)=1} L_{\mathfrak p_1}\Lambda_{\mathfrak p_1} M } \]

The following Adelic Approximation Theorem gives a description of the unit object in $\mathcal C$ as a homotopy limit of a cubical diagram.  For simplicity, following the example just given, we state it in the 1-dimensional case which is given by a homotopy pullback; the general case requires more general homotopy limits.

\begin{theorem} \cite[8.1]{bg}
	Let $\mathcal C$ be an finite-dimensional Noetherian model category with $1$-dimensional Balmer spectrum.  Then the unit object $\mathbb 1$ of $\mathcal C$ is the homotopy pullback of the diagram $\mathbb 1_{\ad} \colon \mathcal P_0[1] \rightarrow \mathcal C$ given by
	\[ \xymatrix{& \prod_{\dim(\mathfrak p_0)=0} L_{\mathfrak p_0} \Lambda_{\mathfrak p_0} \mathbb 1 \ar[d] \\
		\prod_{\dim(\mathfrak p_1)=1} L_{\mathfrak p_1} \Lambda_{\mathfrak p_1} \mathbb 1 \ar[r] & \prod_{\dim(\mathfrak p_0)=0} L_{\mathfrak p_0} \prod_{\dim(\mathfrak p_1)=1} L_{\mathfrak p_1}\Lambda_{\mathfrak p_1} \mathbb 1 } \]
\end{theorem}

We similarly state the following theorem in this simplified case.

\begin{definition} \cite[9.1]{bg} \label{bg9.1}
	Let $\mathcal C$ be a finite-dimensional Noetherian model category with $1$-dimensional Balmer spectrum.  We define the adelic module category $\mathbb 1_{\ad}-\modcat_\mathcal C$ to be the diagram of module categories of the appropriate adelic rings
	\[ \xymatrix{ & \mathbb 1_{\ad}(0)-\modcat_\mathcal C \ar[d] \\
		\mathbb 1_{\ad}(1)-\modcat_\mathcal C \ar[r] & \mathbb 1_{\ad}(0<1)-\modcat_\mathcal C.} \]
	The morphisms are given by extension of scalars functors corresponding to the maps of rings, which are in particular left Quillen, and we can give the category of such diagrams the injective model structure.
\end{definition}

Now, given an object of $\mathcal C$, we can tensor with the diagram $\mathbb 1_{\ad}$ to obtain an object of $\mathbb 1_{\ad}-\modcat_\mathcal C$; given such a diagram, we can take its homotopy limit.

\begin{prop} \cite[9.2]{bg} \label{bg9.2}
	The tensor and limit functors define a Quillen pair
	\[ \mathbb 1_{\ad} \otimes - \colon \mathcal C \leftrightarrows \mathbb 1_{\ad}-\modcat_\mathcal C \colon \lim. \]
\end{prop}

Finally, we come to the application of the homotopy limit construction to this situation.

\begin{theorem} \cite[9.3]{bg}
	Let $\mathcal C$ be an finite-dimensional Noetherian model category.  Then there is a Quillen equivalence between $\mathcal C$ and the homotopy limit of its associated diagram of adelic module categories $\mathbb 1_{\ad}- \modcat_\mathcal C$.
\end{theorem}

Observe that $\mathcal C$ is not assumed to be combinatorial here; see \cite[4.20, 9.3]{bg} for a discussion of how to work around this obstacle.

\section{On not mixing left and right Quillen functors}  \label{mixed}

We now return to the question raised in Remark \ref{whyleft}, namely, whether we can define a homotopy pullback of a diagram of model categories
\[ \mathcal M_1 \overset{F_1}{\rightarrow} \mathcal M_3 \overset{G_2}{\leftarrow} \mathcal M_2, \]
where $F_1$ is a left Quillen functor and $G_2$ is a right Quillen functor.  

In this section, we talk though possible solutions to this question, but ultimately why they are unsatisfactory.  These ideas were developed in several discussions with Yuri Berest.

A first idea for how to mix the kinds of functors is to make the following definition.

\begin{definition}
	Consider a diagram of model categories
	\[ \mathcal M_1 \overset{F_1}{\rightarrow} \mathcal M_3 \overset{G_2}{\leftarrow} \mathcal M_2   \]
	where $F_1$ is a left Quillen functor and $G_2$ is a right Quillen functor with left adjoint $F_2$.  Define the \emph{lax homotopy pullback} $\mathcal L$
	to be the category whose objects are 5-tuples $(x_1, x_2, x_3; u,v)$ where each $x_i$ is an object of $\mathcal M_i$, and 
	\[ u \colon F_1(x_1) \rightarrow x_3, v \colon F_2(x_3) \rightarrow x_2. \]
	A morphism 
	\[ f \colon (x_1, x_2, x_3; u,v) \rightarrow (y_1, y_2, y_3; z, w) \]
	is given by morphisms $f_i \colon x_i \rightarrow y_i$ in $\mathcal M_i$ such that the diagrams
	\[ \xymatrix{F_1(x_1) \ar[r]^-u \ar[d]_{F_1(f_1)} & x_3 \ar[d]^{f_3} & x_2 \ar[d]_{f_2} & F_2(x_3) \ar[d]^{F_2(f_3)} \ar[l]_-v \\
		F_1(y_1) & y_3 & y_2 \ar[r]^-w & F_2(y_3) \ar[l]_-z} \]
	commute.  Give $\mathcal L$ the injective model structure, in which the weak equivalences and cofibrations are defined levelwise.
	
	The \emph{homotopy pullback} is the full subcategory of $\mathcal L$ consisting of objects $(x_1, x_2, x_3;u,v)$ such that $u$ is a weak equivalence in $\mathcal M_3$ and $v$ is a weak equivalence in $\mathcal M_2$.
\end{definition}

The problem here is that we do not get the homotopy limit that we want.  This definition is identical to that of the homotopy limit of the diagram
\[ \mathcal M_1 \overset{F_1}{\rightarrow} \mathcal M_3 \overset{F_2}{\rightarrow} \mathcal M_2. \]
But this homotopy limit is equivalent to the model category $\mathcal M_1$, as the initial object in the diagram, so we do not get something capturing the structure we expect in a homotopy pullback.

However, the above definition is only one possibility.  Perhaps the problem is that we asked for the map $F_2(x_3) \rightarrow x_2$ to be a weak equivalence; namely, asked for a weak equivalence after passing to the adjoint map.  What if we instead ask for the map $x_3 \rightarrow G_2(x_2)$ to be a weak equivalence, and then use its adjoint to define the morphisms in the model category (so that cofibrations behave well)?  So, here is an alternative definition.

\begin{definition}
	Consider a diagram of model categories
	\[ \mathcal M_1 \overset{F_1}{\rightarrow} \mathcal M_3 \overset{G_2}{\leftarrow} \mathcal M_2   \]
	where $F_1$ is a left Quillen functor and $G_2$ is a right Quillen functor with left adjoint $F_2$.  Define the \emph{lax homotopy pullback} $\mathcal L$
	to be the category whose objects are 5-tuples $(x_1, x_2, x_3; u,v)$ where each $x_i$ is an object of $\mathcal M_i$, and 
	\[ u \colon F_1(x_1) \rightarrow x_3, v \colon x_3 \rightarrow G_2(x_2). \]
	A morphism 
	\[ f \colon (x_1, x_2, x_3; u,v) \rightarrow (y_1, y_2, y_3; z, w) \]
	is given by morphisms $f_i \colon x_i \rightarrow y_i$ in $\mathcal M_i$ such that the diagrams
	\[ \xymatrix{F_1(x_1) \ar[r]^-u \ar[d]_{F_1(f_1)} & x_3 \ar[d]^{f_3} & x_2 \ar[d]_{f_2} & F_2(x_3) \ar[d]^{F_2(f_3)} \ar[l]_-{v'} \\
		F_1(y_1) & y_3 & y_2 \ar[r]^-{w'} & F_2(y_3) \ar[l]_-z} \]
	commute, where $v'$ denotes the adjoint to $v$ and similarly for $w'$.  Give $\mathcal M$ the injective model structure, in which the weak equivalences and cofibrations are defined levelwise.
	
	The \emph{homotopy pullback} is the full subcategory of $\mathcal L$ consisting of objects $(x_1, x_2, x_3;u,v)$ such that $u$ and $v$ are weak equivalences in $\mathcal M_3$.
\end{definition}

However, we claim that this definition still does not work.  The first step is to show that there is a model structure on $\mathcal L$ so that the cofibrant objects have $u$ and $v$ weak equivalences, obtained by right Bousfield localization on the injective model structure, via the same kind of proof as used above for Theorem \ref{limmc}.

As in that proof, we can use the combinatorial structure on each of the model categories $\mathcal M_i$ to produce an appropriate generating set of objects $(x_1, x_2, x_3; u, v)$ which respect to which we localize, and for which $u$ and $v$ are weak equivalences.

A key point in that proof, however, is that taking homotopy colimits of objects in this generating set preserves the property that the maps $u$ and $v$ are weak equivalences.  In our new modified setting, when we look at morphisms, they are defined in terms of $v'$, not in terms of $v$.  If the maps $v'$ were weak equivalences, then that property would be preserved.  It does not seem to be the case, however, that homotopy colimits should preserve that $v$ is a weak equivalence.  We cannot even make any assumption about $v'$ being a weak equivalence when $v$ is, so it seems we lose all control over what we know about $v$.	

In summary, we do not seem to have a good way to take homotopy pullbacks of diagrams consisting of both left and right Quillen functors; we expect that the situation for homotopy limits of similarly mixed diagrams is analogously problematic.

\end{document}